\documentclass[a4paper]{amsart}
\usepackage[utf8]{inputenc}
\usepackage{tikz}
\usepackage{soul}

\title{A non-de Finetti theorem for countable Euclidean spaces}

\author{Colin Jahel and Pierre Perruchaud}

\makeatletter
\@namedef{subjclassname@2020}{\textup{2020} Mathematics Subject Classification}
\makeatother

\thanks{C.J. was supported by the Deutsche Forschungsgemeinschaft (DFG, German Research Foundation) -- project number 467967530. P.P. was partially supported by the Luxembourg National Research Fund -- grant O22/17372844/FraMStA}
\usepackage[T1]{fontenc}
\usepackage[osf]{mathpazo}
\usepackage{amssymb}

\usepackage{eucal}

\usepackage{hyperref}
\usepackage[num-plain]{mymacros}

\usepackage{enumitem}
\setlist[enumerate,1]{label=(\roman*), font=\normalfont}

\usepackage[initials, shortalphabetic]{amsrefs}

\begin{document}
\begin{abstract}

The classical de Finetti Theorem classifies the $\mathrm{Sym}(\mathbb N)$-invariant probability measures on $[0,1]^{\mathbb N}$. More precisely it states that those invariant measures are combinations of measures of the form $\nu^{\otimes\mathbb N}$ where $\nu$ is a measure on $[0,1]$. Recently, Jahel--Tsankov generalized this theorem showing that under conditions on $M$, the group $\Aut(M)$ is de Finetti, i.e. $\Aut(M)$-invariant measures on $[0,1]^M$ are mixtures of measures of the form $\nu^{\otimes M}$ where $\nu$ is a measure on $[0,1]$. In this note, we give an example of a non-de Finetti non-Archimedean group.

\end{abstract}

\maketitle

We call a Polish group $G$ non-Archimedean if the identity admits a basis of open subgroups. Such a group can always be represented as the automorphism group $\mathrm{Aut}(M)$ of a model-theoretic countable structure $M$. Using this representation, for every set $Z$, $G$ acts on $Z^M$ via the shift action which is defined as follow: for $x\in Z^M$, $g\in G$ and $a\in M$, then $(g\cdot x)(a)=x(g^{-1}a)$. We are interested in probability measures that are invariant under this type of action, in the sense that for all $g\in G$ and $A\subset Z^M$ measurable, $\mu(g\cdot A)=\mu(A)$. Under reasonable hypotheses, those can be constructed as generalized convex combinations of ergodic measures.

\begin{defn}\index{Ergodic}Let $G$ be a group and $G\actson X$ be an action preserving some probability measure $\nu$ on $X$. The measure $\nu$ is said to be \emph{$G$-ergodic} if for all $A  \subset X$ measurable such that 
\[ \forall g\in G, \ \nu(A \triangle g\cdot A)=0, \]
we have $\nu(A)\in \{0,1\}$.
\end{defn}

In our case, $G=\mathrm{Aut}(M)$ is a non-Archimedean Polish group, and actions on $Z^M$ where $Z$ is topological are continuous. We focus on the relationship between model-theoretic properties of $M$ and dynamical properties of the action $G\actson Z^M$. The classical de Finetti Theorem states that if $M$ is trivial, i.e. every bijection is an isomorphism, then the only ergodic measures on $\{0,1\}^M$ are of the form $\mathrm{Ber}(p)^{\otimes M}$, where $p\in[0,1]$. This motivates our definition of de Finetti groups.
\begin{defn}
Let $M$ be a countable structure. We say that $\Aut(M)$ is \textbf{de Finetti} if for all $Z$ standard Borel the only invariant, ergodic probability measures on $Z^M$ under the shift action are product measures of the form $\lambda^{\otimes M}$, where $\lambda$ is a probability measure on $Z$.
\end{defn}

Similar notions have been discussed in \cite{Tsankov}, \cite{JJ}, \cite{TsankovNewPreprint}, and \cite{BJJ}. One can find in \cite{JT} a wealth of examples, phrased in the language of Fraïssé limits. A Fraïssé class $\mathcal F$ is a collection of finite structures satisfying some compatibility conditions, and the corresponding Fraïssé limit $M$ is some countably infinite structure characterized by the fact that the finite substructures of $M$ are precisely the elements of $\mathcal F$, and that partial isomorphisms $A\to B$ between finite substructures $A,B\subset M$ extend to full automorphisms of $M$, a property that is called ultrahomogeneity. The structure $M$ that we chose to describe $G$ as $\mathrm{Aut}(M)$ can always be taken to be a Fraïssé limit, therefore we lose no generality in stating all results in these terms. Later, we assume familiarity with the notion of Fraïssé classes, limits, and ultrahomogeneity and we refer to \cite{Hodges} for basic notions. The main result of \cite{JT} can be stated as follows, where the model-theoretic properties of $M$ are defined below.

\begin{theorem}[\cite{JT}]
  \label{th:intro:deFinetti}
  Suppose that $M$ is the Fraïssé limit of some class $\mathcal F$.
  
  If $M$ is $\aleph_0$-categorical, transitive, has no algebraicity, and admits weak elimination of imaginaries, then its isomorphism group $G = \Aut(M)$ is de Finetti.
\end{theorem}

In \cite{JT}, it is shown by way of counterexample that one cannot remove any of the hypotheses of transitivity, no algebraicity, or weak elimination of imaginaries. However, until now it was not known whether the condition of $M$ being $\aleph_0$-categorical was superfluous.
The goal of this note is to fill this gap:

\begin{theorem}\label{thm:main}
    There is $M$ a transitive Fra\"iss\'e limit with no algebraicity that admits weak elimination of imaginaries such that $\mathrm{Aut}(M)$ is not de Finetti.
\end{theorem}
In fact, we give a specific example of an ergodic measure which is not a mixture of independent identical distributions.

In the case of non-Archimedean groups, the hypotheses of Theorem \ref{th:intro:deFinetti} rephrase as follows. The definitions here are given from a permutation group perspective and may require some work to prove that they are equivalent their classical model-theoretic formulation.
\begin{enumerate}
\item[1)]$M$ has no algebraicity if for any tuple $\overline{a}\in  M^k$, for any $x\notin \overline{a}$, $G_{\overline{a}} \cdot x$ is infinite, where $G_{\overline{a}}$ denotes the stabilizer of $\overline{a}$ for the action of $\Aut( M)$ on $M$.
\item[2)] $ M$ is $\aleph_0$-categorical if for all $n\in \mathbb N$, $G \actson M^n$ has finitely many orbits.
\item[3)]$M$ has weak elimination of imaginaries if for every proper, open subgroup $V < \Aut(M)$, there exists $k$ and a tuple $\bar a \in M^k$ such that $G_{\overline a} \leq V$ and $[V : G_{\bar a}] < \infty$.
\item[4)] $M$ is transitive if for any $a,b\in M$, there is $g\in \Aut(M)$ such that $g\cdot a=b$.
\end{enumerate}

Our counterexample is constructed as follows. Let $S$ be a dense countable subset of $\mathbb R$. We define the Fra\"iss\'e class $\mathcal{F}$ of finite $S$-metric spaces $(x_1,\ldots,x_n)$ that are subspaces of the unit sphere of $\mathbb R^n$ for some $n$ and such that $(0,x_1,\ldots,x_n)$ is affinely independent in $\mathbb R^n$. The morphisms are given by isometries of the spaces pointed at zero. In the equivalent model theoretic point of view, we have a binary relation for every distance in $S$ and a constant symbol for zero, and the theory contains enough statements to impose that the finite models embed as maximally spanning subsets of the unit sphere, with the exception of zero, which we send to zero. For the rest of this note, we will denote by $M$ the Fra\"iss\'e limit of that class and by $G$ the automorphism group of $M$.

The class $\mathcal{F}$ has been studied by Nguyen Van Th\'e, who proves that this  is a Fra\"iss\'e class, and that it has the strong amalgamation property (see Proposition 7 in \cite{NVT10}), which is equivalent to say that $M$ has no algebraicity. It is easy to see that the Fra\"iss\'e limit of that class is not $\aleph_0$-categorical, as there are infinitely many orbits under the diagonal action of $G$ on $M^2$. In particular, if we prove that $M$ weakly eliminates imaginaries and that there is a non-product $G$-ergodic invariant measure on a product space $Z^M$, then we get Theorem \ref{thm:main}.

We prove 
\begin{prop}
Under these condition the space weakly eliminates imaginaries.
\end{prop}

\begin{proof}
We use the following criterion which is Corollary 7.6 in \cite{Conant}, building on Proposition 5.5.3 in \cite{ThesisKruckman}:
\begin{prop}
Let $M$ be a Fra\"iss\'e limit with no algebraicity, if for all finite $C \subset M$ and $p \in S_1(C)$, any $C$-definable equivalence relation $E(x, y)$ on $p$ is
either equality or trivial, then $M$ weakly eliminates imaginaries
\end{prop}

Accordingly, fix $C\subset M$ finite and some type $p\in S_1(C)$, and let $E$ be some equivalence relation invariant under $G_C$. If $E$ is the equality, then we are done; otherwise, let $x\neq y$ in $p$ be such that we do have $E(x,y)$. Define the relation $\sim$ on $p$ where two points are related if we can go from one to the other using finitely many (possibly zero) jumps whose type over $C$ is either that of $(x,y)$ or of $(y,x)$. It is clearly an equivalence relation, and if $a\sim b$ then we must have $E(a,b)$ since $E$ is invariant by $G_C$. Moreover, by ultrahomogeneity, the type of $(a,b)$ over $C$ depends only on the distance between $a$ and $b$ when they are both in $p$, so whether $a\sim b$ or not depends only on the distance between $a$ and $b$.

We will show two things. First, that there exists an $\varepsilon>0$ such that for every $r<\varepsilon$ in $S$, one may find $z\in p$ such that $|z-x|=r$ and $x\sim z$, which by the previous comment implies that for all $z\in M$ such that $|z-x|<\varepsilon$, $x\sim z$. Second, that $p$ is uniformly connected, i.e. that if $\delta>0$ and $U\subset p$ is a subset such that $B_a(\delta)\subset U$ for all $a\in U$, then either $U=p$ or $U$ is empty. This will show that the class of $x$ for $\sim$ is $p$ itself, hence it will be true for $E$ also and $E$ will be trivial.
\medskip

Let us consider the first assertion. We assume up to isomorphism that $C\cup\{x,y\}$ is isometrically embedded in $\mathbb S^{n+2}\subset\mathbb R^{n+3}$ in such a way that $C$ generates $\mathbb R^n\times\{(0,0,0)\}$ as a vector space. Consider the collection $\tilde p$ of points $a\in\mathbb S^{n+2}$ such that $|a-c|=|x-c|$ for all $c\in C$. Since the inner product of any $a\in\mathbb S^{n+2}$ with a given $c\in C$ is determined by the distance between those points, we know that the projection of $a$ on the space generated by $C$, or equivalently its first $n$ coordinates, are determined by the distances to the points of $C$; in particular, the points of $\tilde p$ all have coordinates of the form $(x_1,\ldots,x_n,a_{n+1},a_{n+2},a_{n+3})$. In the other direction, any point of this form on the sphere must be in $\tilde p$, since the inner product with elements of $C$ does not depend on the last coordinates. This shows that $\tilde p$ is the intersection of $\{(x_1,\ldots,x_n)\}\times\mathbb R^3$ with the unit sphere, and in particular it is a two-dimensional sphere of radius $\rho$ given by $\rho^2=x_{n+1}^2+x_{n+2}^2+x_{n+3}^2>0$. The group of isometries of $\mathbb R^{n+3}$ fixing $C$ acts on $\tilde p$ like the rotation group of the sphere $\tilde p$.

Note first that we cannot have $y=(x_1,\ldots,x_n,-x_{n+1},-x_{n+2},-x_{n+3})$, otherwise $y$ would belong to the vector space generated by $C$ and $x$. This shows that $x$ and $y$ are not antipodal on the sphere $\tilde p$. In particular, the image $x(\theta)$ of $x$ under the rotation of angle $\theta$ around the axis of $\tilde p$ passing through $y$ (the orientation does not matter but is fixed) belongs to the vector space generated by $C$, $x$ and $y$ only if $\theta$ is a multiple of $\pi$, and $x\neq x(\theta)$. Set $\varepsilon=|x(\pi)-x|>0$. If $r<\varepsilon$ is in $S$, then by continuity we can find some $0<\theta<\pi$ such that $|x(\theta)-x|=r$. One can check readily that $C\cup\{x,y,x(\theta)\}$ is in fact an element of $\mathcal F$, so it can be identified with a subset of $M$. Moreover, the rotation sending $C\cup\{x,y\}$ to $C\cup\{x(\theta),y\}$ extends to an isomorphism of $M$ by ultrahomogeneity, and $(x(\theta),y)$ has the type of $(x,y)$ over $C$. This shows that $x\sim y\sim x(\theta)$, and as argued before the function $(a,b)\mapsto\mathbf1_{a\sim b}$ is uniformly continuous.

For later use, note that since $x$ has norm one, $\rho$ could have been described by some function of the distances from $x$ to the points of $C$, hence $\rho$ depends only on $C$ and $p$, not on the particular representation in $\mathbb R^{n+3}$. Moreover, we could have considered more points than just $\{x,y\}$, and concluded by the same arguments that every finite set of $p$ embeds isometrically into a Euclidean sphere of radius $\rho$, where no two points are antipodal.
\medskip

About the uniform connectedness, let $U\subset p$ be a set such that both $U$ and $U^\complement$ are non-empty. Suppose by contradiction that the distance between $U$ and $U^\complement$ is positive. Let $2\rho\sin(\theta/2)$ be this distance, with $0<\theta<\pi$ and $\rho$ defined as above. We can write it this way because any two points of $p$ are at distance less than $2\rho$; $\theta$ is the minimal angle between the vectors on this sphere. Choose $a\in U$, $b\in U^\complement$ such that $d(a,b)<2\rho\sin(\phi/2)$ with $\theta<\phi<\min(2\theta,\pi)$. If we find $z\in p$ such that $|z-a|,|z-b|<2\rho\sin(\phi/4)$, then we will have a contradiction, since then $z$ could not be a member of either $U$ or $U^\complement$. Let us, in the same way as before, consider $C\cup\{a,b\}$ as a subset of $\mathbb S^{n+2}\subset\mathbb R^{n+3}$ in such a way that $C$ generates $\mathbb R^n\times\{(0,0,0)\}$. The set $\tilde p$ of points whose distances to $C$ coincide with that of $a$ is again a sphere of radius $\rho$, and $a$ and $b$ are again not antipodal on this sphere. We want to choose a point $z\in\tilde p$ satisfying the following conditions. We want
\begin{itemize}
    \item the angle between $z$ and $b$ and $z$ and $a$, inside the sphere $\tilde p$, to be less than $\phi/2$;
    \item $z$ not to belong to the vector space generated by $C\cup\{a,b\}$; 
    \item the distances between $z$ and $a$ and $z$ and $b$ to be elements of $S$.
\end{itemize}

The first condition is open, and non-empty since the angle between $a$ and $b$ is less than $\phi$. The second is open and dense, since the intersection of this vector space with the sphere $\tilde p$ is the great circle through $a$ and $b$. The final one is dense in neighborhoods of points $z$ where the circles on $\tilde p$ centered at $a$ and $b$ intersect transversally; this holds unless $z$ lies on the great circle through $a$ and $b$, hence the condition is everywhere dense. The intersection of those conditions is then non-empty, and we found a point $z$ such that $C\cup\{a,b,z\}$ is in $\mathcal F$ and the distances to $a$ and $b$ are less than $2\rho\sin(\phi/4)$. Embedding this space in $M$ appropriately gives the expected contradiction.

\end{proof}

\begin{prop}
   There is a $G$-ergodic invariant measure on $\mathbb{R}^M$ that is not a product measure.
\end{prop}

\begin{proof} First remark that we can isometrically embed $M$ in $\ell^2(\mathbb N)$. We will use the Gaussian space on $\ell^2(\mathbb{N})$, the existence of which is argued in \cite{Janson} Theorem 1.23. A Gaussian space on $\ell^2(\mathbb N)$ is a family of random variables $(\eta_a)$ indexed by $\ell^2(\mathbb N)$ such that for all $a\in \ell^2(M)$ $\eta_a$ has a centered Gaussian distribution and the correlation between $\eta_a$ and $\eta_b$ is $\langle a , b \rangle$ for all $a,b\in \ell^2(\mathbb{N})$. Moreover, this distribution $\mu$ is invariant by isometries. Since such a distribution is unique, it must be invariant by isometries.

There is therefore a family of Gaussian random variables $(\eta_a)_{a\in M}$ such that the correlation between $\eta_a$ and $\eta_b$ is $\left\langle a,b \right\rangle$ for $a,b\in M$. This measure is a non product $\Aut(M)$-invariant measure on $\mathbb R^M$.

All that remains to do is to prove ergodicity of the associated measure. Let $A\subset R^M$ be such that for all $g\in G$, $\mu(A\triangle g \cdot A)=0$. Let $\varepsilon>0$, there is a cylinder, i.e. a measurable set depending only on the values of $\eta$ at finitely many elements of $M$, $B=B(x_1,\ldots,x_n)$, where $x_1,\ldots,x_n\in M$, such that
\[\mu(A\triangle B)\leq \varepsilon.\] 

Let us first assume that $S$ is a ring closed under taking square roots, as it is an easier version of the full argument. In this case, there is $(y_1,\ldots,y_n)\in M$ orthogonal and isometric to $(x_1,\ldots,x_n)$, therefore there is $g\in \mathrm{Aut}(M)$ sending $(x_1,\ldots,x_n)$ to $(y_1,\ldots,y_n)$. This implies that $g\cdot B $ is independent from $B$ and we get

\begin{align*}
    |\mu(A)-\mu(A)^2|&=|\mu(A\cap g\cdot A)-\mu(A)^2| \\ &\leq |\mu(A)^2-\mu(B\cap g \cdot B)|+|\mu(B\cap g \cdot B)-\mu(A \cap g \cdot A)| \\
    &\leq |\mu(A)^2-\mu(B)^2| + |\mu(B\triangle A)|  \\
    &+|\mu(g\cdot B\triangle g \cdot A)| \\
    &\leq 4\varepsilon.
\end{align*}

If $S$ is not closed under taking the square root, for any $\varepsilon >0$ and $k\in \N$, we can find $(z^k_1,\ldots,z^k_n)\in M$ isometric to $(x_1,\ldots,x_n)$ such that $\sup_i |\langle x_i,z^k_j \rangle | \leq 1/k$. Such a family exists since $S$ is dense and being a Euclidean space with no affine dependence is an open condition on the distance function.
As $k$ goes to infinity, the covariance of $(\eta_{x_1},\ldots,\eta_{x_n},\eta_{z^k_1},\ldots,\eta_{z^k_n})$ converges to that of a pair of independent vectors distributed as $(\eta_{x_1},\ldots,\eta_{x_n})$. Since the vectors are Gaussian, we know that the convergence actually holds in total variation distance (say for instance using Scheffé's lemma), and in particular $\mu(B\cap g_k\cdot B)$ converges to $\mu(B)^2$.

The rest of the proof is just as in the case closed under taking square roots.
\end{proof}

This example also gives a counterexample to the removal of the $\aleph_0$-categorical hypothesis to  the main theorem of \cite{JT}, which states

\begin{theorem} Let $M$ be a transitive, $\aleph_0$-categorical Fra\"iss\'e limit with no algebraicity that admits weak elimination of imaginaries. Consider the action $\Aut(M) \actson \LO(M)$ induced on the linear orders of $M$. Then exactly one of the following holds:
  \begin{enumerate}
  \item The action $\Aut(M) \actson \LO(M)$ has a fixed point (i.e., there is a definable linear order on $M$);
  \item The action $\Aut(M) \actson \LO(M)$ has a unique invariant measure.
  \end{enumerate}
\end{theorem}

Indeed our example satisfies the hypotheses, except for $\aleph_0$-categoricity, but not the conclusion: it does not have a definable linear ordering (by ultrahomogenous extension of transpositions), and the random linear ordering induced by ordering the Gaussian values is not the uniform one, therefore there are at least two random $G$-invariant linear orderings. This order is fully supported, since the event $\{\eta_{x_1}<\cdots<\eta_{x_k}\}$ corresponds to outcomes when a non-degenerate Gaussian vector takes values into a non-empty open set.

\bibliographystyle{alpha}
\bibliography{bib}

\begin{bibdiv}
\begin{biblist}

\bib{BJJ}{article}{
      author={Barritault, Rémi},
      author={Jahel, Colin},
      author={Joseph, Matthieu},
       title={Unitary representations of the isometry groups of urysohn
  spaces},
        date={2024},
     journal={preprint},
      eprint={2410.01725},
         url={https://arxiv.org/abs/2410.01725},
}

\bib{Conant}{article}{
      author={Conant, Gabriel},
       title={Neostability in countable homogeneous metric spaces},
        date={2017},
        ISSN={0168-0072,1873-2461},
     journal={Ann. Pure Appl. Logic},
      volume={168},
      number={7},
       pages={1442\ndash 1471},
         url={https://doi.org/10.1016/j.apal.2017.01.008},
      review={\MR{3638897}},
}

\bib{Hodges}{book}{
      author={Hodges, Wilfrid},
       title={Model theory},
      series={Encyclopedia of Mathematics and its Applications},
   publisher={Cambridge University Press},
        date={1993},
}

\bib{Janson}{book}{
      author={Janson, Svante},
       title={Gaussian {H}ilbert spaces},
      series={Cambridge Tracts in Mathematics},
   publisher={Cambridge University Press, Cambridge},
        date={1997},
      volume={129},
        ISBN={0-521-56128-0},
         url={https://doi.org/10.1017/CBO9780511526169},
      review={\MR{1474726}},
}

\bib{JJ}{article}{
      author={Jahel, C.},
      author={Joseph, M.},
       title={Stabilizers for ergodic actions and invariant random expansions
  of non-archimedean polish groups},
        date={2023},
     journal={preprint},
      eprint={2307.06253},
}

\bib{JT}{article}{
      author={Jahel, Colin},
      author={Tsankov, Todor},
       title={{Invariant measures on products and on the space of linear
  orders}},
        date={2022},
     journal={J. Ec. Polytech. Math.},
      volume={9},
       pages={155\ndash 176},
}

\bib{ThesisKruckman}{thesis}{
      author={Kruckman, A.},
       title={Infinitary limits of finite structures},
        type={Ph.D. Thesis},
        date={2016},
}

\bib{NVT10}{article}{
      author={{Nguyen Van Th\'{e}}, L.},
       title={Structural {R}amsey theory of metric spaces and topological
  dynamics of isometry groups},
        date={2010},
        ISSN={0065-9266},
     journal={Mem. Amer. Math. Soc.},
      volume={206},
      number={968},
       pages={x+140},
         url={https://doi.org/10.1090/S0065-9266-10-00586-7},
      review={\MR{2667917}},
}

\bib{TsankovNewPreprint}{article}{
      author={Tsankov, T.},
       title={Non-singular and probability measure-preserving actions of
  infinite permutation groups},
        date={2024},
     journal={preprint},
}

\bib{Tsankov}{article}{
      author={Tsankov, T.},
       title={Unitary representations of oligomorphic groups},
        date={2012},
     journal={Geom. Funct. Anal.},
      volume={22},
      number={2},
       pages={528\ndash 555},
}

\end{biblist}
\end{bibdiv}

			\bigskip
			\footnotesize
			
			\noindent C.~Jahel, \textsc{Institut fur Algebra, Technische Universität Dresden, Germany}\par\nopagebreak\noindent
			\textit{E-mail address: }\texttt{colin.jahel@tu-dresden.de}
			
			\medskip
			
			\noindent P.~Perruchaud, \textsc{Université du Luxembourg, Département de Mathématiques, Grand Duché du Luxembourg}\par\nopagebreak\noindent
			\textit{E-mail address: }\texttt{pierre.perruchaud@uni.lu}
	
			\medskip

\end{document}